\documentclass[12pt,twoside,reqno]{amsart}
\usepackage[colorlinks=true,citecolor=blue]{hyperref}
\usepackage{mathptmx, amsmath, amssymb, amsfonts, amsthm, mathptmx, enumerate, color}
\usepackage{graphicx}
\usepackage{epstopdf}

\newtheorem{theorem}{Theorem}
\newtheorem{corollary}{Corollary}[section]

\newtheorem{proposition}{Proposition}[section]
\theoremstyle{definition}
\newtheorem{definition}{Definition}[section]

\newtheorem{remark}{Remark}[section]

\numberwithin{equation}{section}
\begin{document}
\setcounter{page}{1}

\def\real{\mathbb{R}}
\newtheorem{Problem}{Problem}[section]
\newcommand{\wto}{ \ \stackrel{w} {\longrightarrow} \ }

\vspace*{1.0cm}
\title[Coupled variational inequalities]
{Existence of solution to a new class of coupled variational-hemivariational inequalities}
\author[Yunru Bai, Stanislaw Mig\'orski, Van Thien Nguyen, Jianwen Peng]{Yunru Bai$^{1,*}$, Stanislaw Mig\'orski$^2$, Van Thien Nguyen${^3}$, Jianwen Peng$^4$}
\maketitle
\vspace*{-0.6cm}

\begin{center}
{\footnotesize {\it

$^1$School of Science, Guangxi University of Science and Technology,
Liuzhou 545006, Guangxi, China.  \\

$^2$College of Applied Mathematics, Chengdu University of Information Technology, Chengdu 610225, Sichuan, P.R. China, and Jagiellonian University in Krakow, Faculty 
of Mathematics and Computer Science, \\ ul. Lojasiewicza~6, 30348 Krakow, Poland. \\

$^3$Departement of Mathematics, FPT University, Education zone, Hoa Lac High Tech Park, \\ Km29 Thang Long highway, Thach That ward, Hanoi, Vietnam. \\

$^4$School of Mathematics Science, Chongqing Normal University, Chongqing 401331, P.R. China.

}}\end{center}

\vskip 4mm {\small \noindent {\bf Abstract.}
The objective of this paper is to introduce and
study a complicated nonlinear system, called coupled variational-hemivariational inequalities, which is described by a highly nonlinear coupled system of inequalities on Banach spaces. We establish the nonemptiness and compactness of the solution set to the system.
We apply a new method of proof based on a multivalued version of the Tychonoff fixed point principle in a Banach space combined with the generalized monotonicity arguments, and elements of the nonsmooth analysis. Our results improve and generalize some earlier theorems obtained for a very particular form of the system.

}

\renewcommand{\thefootnote}{}
\footnotetext{ $^*$Corresponding author.
\par
E-mail addresses: 
yunrubai@163.com (Y. Bai), 
stanislaw.migorski@uj.edu.pl (S. Mig\'orski),
Thiennv15@fe.\-edu.vn (V. T. Nguyen),
jianwenpengcqu@163.com (J. Peng)
\par
Received February 10, 2022; Accepted March 7, 2022. }

\section{Introduction}

In this paper we study the existence of solution to a new class of systems of two nonlinear coupled variational-hemivariational inequalities with constraints.
Each inequality involves a nonlinear operator, the generalized (Clarke) directional derivative of a locally Lipschitz function, a convex potential,
and a constraint set.
The main feature of the system is a strong coupling which appears in the nonlinear operators and the generalized directional derivatives.
Our results concern existence and compactness of the solution set to the system, and generalize the results obtained very recently in~\cite{Liu-Yang-Zeng-Zhao-2021-JOTA} by using
a different method.

To introduce the problem we need the following functional framework which will be used throughout the paper.
Let $(V,\|\cdot\|_V)$ and $(E,\|\cdot\|_E)$ be real reflexive Banach spaces, and $C\subset V$ and $D\subset E$ be nonempty, closed and convex sets.
We are given two nonlinear operators
$A\colon E\times V\to V^*$ and
$B\colon V\times E\to E^*$, two convex functions $\psi\colon V\to \overline\real:=\real\cup\{+\infty\} $ and $\theta\colon E\to \overline\real$, two nonlinear functions $J\colon Z_1\times X\to \real$ and $H\colon Z_2\times Y\to \real$ (which are locally Lipschitz continuous with respect to their second variables), four linear operators
$\gamma_1\colon V\to X$,
$\gamma_2\colon E\to Y$,
$\delta_1\colon E\to Z_1$ and
$\delta_2\colon V\to Z_2$, and two elements
$h\in V^*$ and $l\in E^*$.
The system of two coupled nonlinear variational-hemivariational inequalities reads as follows.
\begin{Problem}\label{problems1}
	Find $u\in C$ and $w\in D$ satisfying the following inequalities
	\begin{equation}\label{eqn1}
	\langle A(w,u),v-u\rangle_V+J^0(\delta_1 w,\gamma_1 u;\gamma_1 (v-u))+\psi(v)-\psi(u)\ge \langle h,v-u\rangle_V
	\end{equation}
	for all $v\in C$, and
	\begin{equation}\label{eqn2}
	\langle B(u,w),z-w\rangle_E+H^0(\delta_2 u,\gamma_2 w;\gamma_2 (z-w))+\theta(z)-\theta(w)\ge \langle l,z-w\rangle_E
	\end{equation}
	for all $z\in D$.
\end{Problem}

It should be mentioning that Problem~\ref{problems1} is new and contains many challenging and important problems as its special cases.
We point out below several interesting particular cases of Problem~\ref{problems1}.
\begin{itemize}
\item[(i)] 
	If $J$ and $H$ are independent of their first variables, respectively, then Problem~\ref{problems1}  reduces to the following coupled system:
	find $(u,w)\in C\times D$ such that
	\begin{equation}\label{prob1a}
	\langle A(w,u),v-u\rangle_V+J^0(\gamma_1 u;\gamma_1 (v-u))+\psi(v)-\psi(u)\ge \langle h,v-u\rangle_V
	\end{equation}
	for all $v\in C$, and
	\begin{equation}\label{prob1b}
	\langle B(u,w),z-w\rangle_E+H^0(\gamma_2 w;\gamma_2 (z-w))+\theta(z)-\theta(w)\ge \langle l,z-w\rangle_E
	\end{equation}
	for all $z\in D$.
	This kind of coupled inequalities  (\ref{prob1a})--(\ref{prob1b})
	has not been studied in the literature.
	\item[(ii)]
	When $\psi=\theta\equiv0$, then Problem~\ref{problems1} takes the following form of two coupled hemivariational inequalities: find $(u,w)\in C\times D$ satisfying
	\begin{equation}\label{prob2a}
	\langle A(w,u),v-u\rangle_V+J^0(\delta_1 w,\gamma_1 u;\gamma_1 (v-u)) \ge \langle h,v-u\rangle_V
	\end{equation}
	for all $v\in C$, and
	\begin{equation}\label{prob2b}
	\langle B(u,w),z-w\rangle_E+H^0(\delta_2 u,\gamma_2 w;\gamma_2 (z-w)) \ge \langle l,z-w\rangle_E
	\end{equation}
	for all $z\in D$.
	To the best of our knowledge, there is no results available in the literature for the system (\ref{prob2a})--(\ref{prob2b}).
	\item[(iii)]
	If $J\equiv0$ and $H\equiv0$, then  Problem~\ref{problems1} becomes the following system of coupled variational inequalities:
	find $(u,w)\in C\times D$ satisfying
	\begin{equation}\label{eqn003}
	\langle A(w,u),v-u\rangle_V+\psi(v)-\psi(u)\ge \langle h,v-u\rangle_V
	\end{equation}
	for all $v\in C$, and
	\begin{equation}\label{eqn004}
	\langle B(u,w),z-w\rangle_E+\theta(z)-\theta(w)\ge \langle l,z-w\rangle_E
	\end{equation}
	for all $z\in D$.
	This system has been considered and investigated
	in~\cite{Liu-Yang-Zeng-Zhao-2021-JOTA} in which the authors applied the Kakutani-Ky Fan fixed point theorem for multivalued operators to prove the existence of solutions to system (\ref{eqn003})--(\ref{eqn004}).
	In this paper, in contrast to~\cite{Liu-Yang-Zeng-Zhao-2021-JOTA}, we give a new proof which is based on a multivalued version of the Tychonoff fixed point principle in a Banach
	space combined with the theory of nonsmoth analysis, generalized monotonicity arguments and the Minty approach.
	\item[(iv)]
	When $\theta\equiv0$, $H\equiv0$ and $D=E$, then Problem~\ref{problems1} can be reformulated
	as the following variational-hemivariational inequality subjected to a nonlinear equation constraint:
	find $(u,w)\in C\times D$ such that
	\begin{eqnarray*}
		\left\{\begin{array}{lll}
			\langle A(w,u),v-u\rangle_V+J^0(\delta_1w,\gamma_1 u;\gamma_1 (v-u))+\psi(v)-\psi(u)\ge \langle h,v-u\rangle_V, \\[1mm]
			\hspace{8.0cm}
			\mbox{for all} \ v \in C, \\
			B(u,w)=l.
		\end{array}\right.
	\end{eqnarray*}
	\item[(v)]
	Assume that $\psi\equiv0$, $\theta\equiv0$, $J\equiv0$, $H\equiv0$, $C=V$ and $D=E$.
	Then Problem~\ref{problems1} is equivalent to the following nonlinear system of two coupled equations: find $(u,w)\in C\times D$ such that
	\begin{eqnarray*}
		\left\{\begin{array}{lll}
			A(w,u)=h, \\[2mm]
			B(u,w)=l.
		\end{array}\right.
	\end{eqnarray*}
	\item[(vi)]
	Suppose that $\theta\equiv0$, $H\equiv 0$, $D=E$ and $B$ is independent of its first variable.
	Then Problem~\ref{problems1} can be reformulated as the following parameter control system driven by a variational-hemivariational inequality:
	find $u\in C$ and $w\in W$ such that
	\begin{equation*}
	\langle A(w,u),v-u\rangle_V+J^0(\delta_1w,\gamma_1 u;\gamma_1 (v-u))+\psi(v)-\psi(u)\ge \langle h,v-u\rangle_V
	\end{equation*}
	for all $v\in C$,
	where the admissible set $W$ is defined by
	$W:=\{w\in E \mid B(w)=l\}$.
	\item[(vii)]
	If $A$ and $J$ are independent of their first variables, $B\equiv0$, $\theta\equiv0$, $H\equiv0$ and $l=0$, then Problem~\ref{problems1} reduces to the following elliptic variational-hemivariational inequality: find $x\in C$ such that
	\begin{equation}\label{NEW1}
	\langle Au, v-u \rangle_V +
	J^0(\gamma_1 (u); \gamma_1 (v-u))+\psi(v)-\psi(u)\ge \langle h,v-u\rangle_V
	\end{equation}
	for all $v\in C$.
\end{itemize}

The variational-hemivariational inequalities of the form (\ref{NEW1}) have been studied
from various perspectives. For example,
the results on noncoercive hemivariational inequalities can be found
in~\cite{Chadli-Schaible-Yao-JOTA-2004} where the equilibrium problems have been
employed, and in~\cite{Liu-Liu-Wen-Yao-Zeng-2020-AMO} where an application
to contact problems in mechanics were treated.
Several classes of variational-hemivariational and hemivariational inequalities that
model problems in contact mechanics have been also studied
in~\cite{Han-Migorski-Sofonea-2014-SIMA,smo1}.
The nonconvex star-shaped constraints sets in evolution hemivariational inequalities
have been studied in~\cite{Gasinski-Liu-Migorski-Peng}, and singular perturbations
of inequality problems were analyzed in~\cite{WeiminHanSIAM2020}.
Optimal control problems and inverse problems for the aforementioned inequalities
have been investigated in~\cite{MKZ2020IV,ZengMigorskiKhanSICON2020}.
The elliptic variational-hemivariational inequalities have been treated in~\cite{liu4,liu2},
differential hemivariational inequalities in~\cite{MigorskiZeng2018JOGO},
and related double phase obstacle problems were considered
in~\cite{Zeng-Bai-Gasinski-Winkert-CVPDEs-2020,Zeng-Bai-Gasinski-Winkert-ANONA-2021}. For other recent results on
hemivariational inequalities, we refer, for example, to~\cite{liumotreanuzengSIOP2021,LiuMotreanuZengCVPDE2019,Ovcharova-Gwinner-JOTA-2014,Tang-Huang-JOGO-2013,Xiao-Huang-JOTA-2011,ZengRadulescuWinkert,ZengBaiGasinskiWinkertCVPDE2020} and the references therein.

The rest of the paper is organized as follows.
In Section~\ref{Section2} we recall a preliminary material needed in the sequel.
Section~\ref{Section3} is devoted to state the hypotheses on the data of Problem~\ref{problems1},
and to deliver the main results of this paper which contain the nonemptiness and compactness of the solution set to Problem~\ref{problems1}.

\section{Mathematical Background}\label{Section2}
In this section, we recall a necessary preliminary material which will be used throughout the paper.
More details can be found in~\cite{barbu,DMP1,DMP2,leszek1,smo1}.

Let $(E,\|\cdot\|_E)$ be a Banach space,
$E^*$ be its dual space, and $\langle\cdot,\cdot\rangle_E$ denote the duality brackets between $E^*$ and $E$.
We adopt the symbols "$\wto$" and "$\to$"
to symbolize the weak convergence and the strong convergence in various spces, respectively.

We recall definitions and properties of upper semicontinuous multivalued ope\-ra\-tors.
\begin{definition}\label{defusc}
	Let $Y$ and $Z$ be topological spaces,
	$D\subset Y$ be a nonempty set, and
	$G\colon Y\to 2^Z$ be a multivalued map.
	\begin{description}
		\item[(i)] The map $G$ is called upper semicontinuous (u.s.c., for short) at $y\in Y,$ if for each open set $O\subset Z$ such that $G(y)\subset O$,  there exists a neighborhood $N(y)$ of $y$ satisfying $G(N(y)):=\cup_{z\in
			N(y)}G(z)\subset O$. If it holds for each $y\in D$, then $G$
		is called to be upper semicontinuous in $D$.
		\item[(ii)] The map $G$ is closed at $y\in Y$, if for every sequence $\{(y_n,z_n)\}\subset \mbox{\rm Gr} (G)$ satisfying  $(y_n,z_n)\to(y,z)$ in $Y\times Z$, it holds $(y,z)\in \mbox{\rm Gr} (G)$, where
		$\mbox{\rm Gr}(G)$ is the graph of the map
		$G$ defined by
		\begin{eqnarray*}
			\mbox{\rm Gr}(G):=\{(y,z)\in Y\times Z \mid z\in G(y)\}.
		\end{eqnarray*}
		If it holds for each $y\in Y$, then $G$
		is called to be closed (or $G$ has a closed graph).
	\end{description}
\end{definition}

Let $X_1$ and $X_2$ be two Banach spaces. A multivalued map $F\colon X_1\to 2^{X_2}$ is called sequentially weakly-weakly closed, if $F$ is sequentially closed from $X_1$ endowed with the weak topology into the subsets of $X_2$ with the weak topology.

The following result provides two useful criteria for
the upper semicontinuity of a multivalued map.
\begin{proposition}\label{uscproposition}
	Let $F\colon X\to 2^Y$ with $X$ and $Y$ topological spaces. The following conditions are equivalent:
	\begin{description}
		\item[{\rm(i)}] $F$ is upper semicontinuous.
		\item[{\rm(ii)}] For each closed set $C\subset Y$,
		$F^{-}(C):=\{x\in X\mid F(x)\cap C\neq\emptyset\}$ is closed in $X$.
		\item[{\rm(iii)}] For each open set $O\subset Y$,
		$F^{+}(O):=\{x\in X\mid F(x)\subset O\}$
		is open in $X$.
	\end{description}
\end{proposition}

The following definitions provide useful
notions from the theory of nonsmooth analysis.
\begin{definition}
	Let $V$ be a reflexive Banach space, $\psi\colon V\to\overline\real$ be a proper, convex and l.s.c. function, and $A\colon V\to 2^{V^*}$ be a multivalued operator. The operator $A$ is called to be
	\begin{enumerate}
		\item[{\rm(i)}]  $\psi$-pseudomonotone, if for any $u$, $v\in V$ fixed, there exists an element
		$u^*\in Au$ such that
		\begin{equation*}
		\langle u^*,v-u\rangle_X+\psi(v)-\psi(u)\ge 0,
		\end{equation*}
		then we have
		\begin{equation*}
		\langle v^*,v-u\rangle_X+\psi(v)-\psi(u)\ge 0
		\end{equation*}
		for all $v^*\in A(v)$.
		\item[{\rm(ii)}] stable $\psi$-pseudomonotone with respect to the set $W\subset V^*$, if $A$ and $V\ni u\mapsto Au-w\subset V^*$ are $\psi$-pseudomonotone for each $w\in W$.
	\end{enumerate}
\end{definition}

\begin{definition}\label{SUB} Let $X$ be a Banach space with the dual space $X^*$ and norm
	$\|\cdot \|_X.$  A function $J\colon X\to \real$ is called to be a locally Lipschitz continuous at $u\in X$, if there are a neighborhood $N(u)$ of $u$ and a constant $L_u>0$ with
	\begin{equation*}
	|J(w)-J(v)|\le L_u\|w-v\|_X
	\ \ \mbox{for all}\ \ w, v \in N(u).
	\end{equation*}
	Given a locally Lipschitz function $J \colon X \to \real$, the generalized (Clarke) directional derivative of $J$ at the point $u\in X$ in the direction $v\in X$, denoted by $J^0 (u; v)$, is defined by
	\begin{equation*}\label{defcalark}
	J^0(u;v) = \limsup
	\limits_{\lambda\to 0^{+}, \, w\to u} \frac{J(w+\lambda v)-J(w)}{\lambda}.
	\end{equation*}
	Further, the generalized subgradient of $J \colon X \to \mathbb{R}$ at $u\in X$ is given by
	\begin{equation*}
	\partial J(u) = \{\, \xi\in X^{*} \mid J^0 (u; v)\ge
	\langle\xi, v\rangle_{X^*\times X} \ \ \mbox{\rm for all} \ \ v \in X \, \}.
	\end{equation*}
\end{definition}

The generalized subgradient and generalized directional derivative of a locally Lipschitz function enjoy nice properties and rich calculus. Here, we summarize some basic results (see for example~\cite[Proposition 3.23]{smo1}).
\begin{proposition}\label{subdiff}
	Let $J \colon X \to \real$ be a locally Lipschitz function.
	Then
	\begin{description}
		\item[{\rm (i)}] For every $x \in X$, the function
		$X \ni v \mapsto J^0(x;v) \in \real$ is positively  homogeneous and subadditive, i.e.,
		$J^0(x; \lambda v) = \lambda J^0(x; v)$ for all
		$\lambda \ge 0$, $v\in X$,
		and $J^0 (x; v_1 + v_2) \le
		J^0(x; v_1) + J^0(x; v_2)$ for all $v_1$, $v_2 \in
		X$, respectively.
		\item[{\rm (ii)}]
		For every $x\in X$, the set $\partial J(x)$ is a nonempty, convex, and weakly* compact subset of $X^*$, which is bounded by the Lipschitz constant $L_x>0$ of $J$ near $x$.
		\item[{\rm (iii)}]
		The graph of the generalized subgradient operator $\partial J$ of $J$ is closed in $X\times (w^*\mbox{--}X^*)$ topology, i.e.,
		if $\{x_n\}\subset X$ and $\{\xi_n\}\subset X^*$ are sequences such that $\xi_n\in \partial J(x_n)$ and $x_n\to x$ in $X$, $\xi_n \wto \xi$ in $X^*$, then $\xi\in \partial J(x)$, where, recall, $(w^*\mbox{--}X^*)$ denotes the space
		$X^*$ equipped with weak* topology.
	\end{description}
\end{proposition}

We end the section by recalling a multivalued version
of the Tychonoff fixed point principle in a Banach
space, its proof can be found in~\cite[Theorem~8.6]{GranasDugundji}.
\begin{theorem}\label{TychonoffFPT}
	Let $C$ be a bounded, closed and convex subset
	of a reflexive Banach space $E$,
	and $S \colon  C \to 2^C$ be a multivalued map such that
	\begin{description}
		\item[\rm (i)] $S$ has bounded, closed and convex values,
		\item[\rm (ii)] $S$ is weakly-weakly u.s.c..
	\end{description}
	Then $S$ has a fixed point in $C$.
\end{theorem}

\section{Main Results}\label{Section3}

This section is devoted to the main results of the paper which include the nonemptiness and compactness of the solution set to Problem~\ref{problems1}.

Before we state and prove the results, we make the following hypotheses on the data of Problem~\ref{problems1}.

\medskip
\noindent ${\underline{H(0)}}$: $C$ and $D$ are nonempty, closed and convex subsets of $V$ and $E$, respectively.

\medskip

\noindent  ${\underline{H(1)}}$: $h\in V^*$ and $l\in E^*$.

\medskip

\noindent  ${\underline{H(2)}}$: $\gamma_1\colon V\to X$, $\gamma_2\colon E\to Y$, $\delta_1\colon E\to Z_1$ and $\delta_2\colon V\to Z_2$ are bounded, linear and compact.

\medskip

\noindent  ${\underline{H(\psi)}}$: $\psi\colon
V\to \overline\real$ is a convex and lower semicontinuous function such that

dom$(\psi):=\{u\in V \mid \psi(u)<+\infty\}\cap C\neq\emptyset$.

\medskip

\noindent  ${\underline{H(\theta)}}$: $\theta\colon
E\to \overline\real$ is a convex and lower semicontinuous function such that

dom$(\theta):=\{w\in E \mid \theta(w)<+\infty\}\cap D\neq\emptyset$.

\medskip

\noindent ${\underline{H(J)}}$: $J\colon Z_1\times X\to\real$ is such that
\begin{itemize}
	\item[(i)] for every $w\in Z_1$, $X\ni u\mapsto J(w,u)\in\real$ is locally Lipschitz continuous.
	\item[(ii)] there exists a constant $c_J\ge 0$ such that
	\begin{equation*}
	\|\xi\|_{X^*}\le c_J\left(1+\|u\|_X+\|w\|_{Z_1}\right)
	\end{equation*}
	for all $\xi\in \partial J(w,u)$, $u\in X$ and $w\in Z_1$.
	\item[(iii)] the following inequality is valid
	\begin{equation*}
	\limsup_{n\to\infty} J^0(w_n,u;v)\le J^0(w,u;v),
	\end{equation*}
	whether $u$, $v$ and sequence  $\{w_n\}\subset Z_1$ is such that
	\begin{equation*}
	w_n\to w\mbox{ in $Z_1$ as $n\to\infty$}
	\end{equation*}
	for some $w\in Z_1$.
\end{itemize}

\medskip

\noindent ${\underline{H(H)}}$: $H\colon Z_2\times Y\to\real$ is such that
\begin{itemize}
	\item[(i)] for every $u\in Z_2$, $Y\ni w\mapsto H(u,w)\in\real$ is locally Lipschitz continuous.
	\item[(ii)] there exists a constant $c_H\ge 0$
	such that
	\begin{equation*}
	\|\eta\|_{Y^*}\le c_H\left(1+\|u\|_{Z_2}+\|w\|_Y\right)
	\end{equation*}
	for all $\eta\in \partial H(u,w)$, $u\in Z_2$ and $w\in Y$.
	\item[(iii)] the following inequality is valid
	\begin{equation*}
	\limsup_{n\to\infty} H^0(u_n,w;z)\le H^0(u,w;z),
	\end{equation*}
	whether $w$, $z\in Y$ and sequence  $\{w_n\}\subset Z_2$ is such that
	\begin{equation*}
	w_n\to w\mbox{ in $Z_2$ as $n\to\infty$}
	\end{equation*}
	for some $w \in Z_2$.
\end{itemize}

\medskip

\noindent ${\underline{H(A)}}$: $A\colon E\times V\to V^*$ satisfies the following conditions:
\begin{itemize}
	\item[(i)] for any $w\in E$ and $v$, $u\in V$, the following inequality holds
	\begin{equation*}
	\limsup_{\lambda\to0}\langle A(w,t v+(1-t)u),v-u\rangle_V\le\langle A(w,u),v-u\rangle_V.
	\end{equation*}
	\item[(ii)] for any $w\in E$ fixed, the multivalued mapping $V\ni u\mapsto A(w,u)+\gamma_1^*\partial J(\delta_1w,\gamma_1u)\subset V^*$ is stable $\psi$-pseudomonotone with respect to $\{h\}$.
	\item[(iii)] if $\{w_n\}\subset E$ and $\{u_n\}\subset V$ are such that
	\begin{equation*}
	w_n\wto w\mbox{ in $E$ and $u_n\wto u$ in $V$ as $n\to\infty$}
	\end{equation*}
	for some $(w,u)\in E\times V$, then we have
	\begin{equation*}
	\limsup_{n\to\infty}\langle A(w_n,v),v-u_n\rangle_V\le \langle A(w,v),v-u\rangle_V.
	\end{equation*}
	\item[(iv)]   the following growth condition is satisfied
	\begin{equation*}
	\|A(w,u)\|_{V^*}\le b_A(1+\|u\|_V+\|w\|_E)
	\end{equation*}
	for all $(w,u)\in E\times V$ with some $b_A>0$.
	\item[(v)] there exists a function $r_A\colon \real_+\times \real_+\to \real$ such that
	\begin{equation*}
	\langle A(w,u),u\rangle_V
	-J^0(\delta_1 w,\gamma_1u;-\gamma_1u)\ge r_A(\|u\|_V,\|w\|_E)\|u\|_V\mbox{ for all $u\in V$ and $w\in E$},
	\end{equation*}
	and
	\begin{itemize}
		\item[$\bullet$] for every nonempty and bounded set $O\subset \real_+$, we have $r_A(t,s)\to +\infty$ as $t\to +\infty$ for all $s\in O$,
		\item[$\bullet$] for any constants $c_1$, $c_2\ge 0$, it holds $r_A(t,c_1t+c_2)\to +\infty$ as $t\to +\infty$,
		\item[$\bullet$] for sequences $\{s_n\}\subset \real_+$ and $\{t_n\}\subset$  such that
		\begin{equation*}
		s_n\to +\infty,\mbox{ $t_n\to+\infty$ and $\frac{t_n}{s_n}\to 0$ as $n\to\infty$},
		\end{equation*}
		we have
		\begin{equation*}
		r_A(s_n,t_n)\to +\infty\mbox{ as $n\to\infty$}.
		\end{equation*}
	\end{itemize}
\end{itemize}

\medskip

\noindent ${\underline{H(B)}}$: $B\colon V\times E\to E^*$ satisfies the following conditions:
\begin{itemize}
	\item[(i)] for any $u\in V$ and $z$, $w\in E$,
	it holds
	\begin{equation*}
	\limsup_{\lambda\to0}\langle B(u,t z+(1-t)w),z-w\rangle_E\le\langle B(u,w),z-w\rangle_E.
	\end{equation*}
	\item[(ii)] for each $u\in V$ fixed, the multivalued mapping $E\ni w\mapsto B(u,w)+\gamma_2^*\partial H(\delta_2u,\gamma_2w)\subset E^*$ is stable $\theta$-pseudomonotone with respect to $\{l\}$.
	\item[(iii)] if $\{w_n\}\subset E$ and $\{u_n\}\subset V$ are such that
	\begin{equation*}
	w_n\wto w\mbox{ in $E$ and $u_n\wto u$ in $V$ as $n\to\infty$}
	\end{equation*}
	for some $(w,u)\in E\times V$, then we have
	\begin{equation*}
	\limsup_{n\to\infty}\langle B(u_n,z),z-w_n\rangle_E\le \langle B(u,z),z-w\rangle_E.
	\end{equation*}
	\item[(iv)]  the following growth condition is satisfied
	\begin{equation*}
	\|B(u,w)\|_{E^*}\le b_B(1+\|u\|_V+\|w\|_E)
	\end{equation*}
	for all $(w,u)\in E\times V$ with some $b_B>0$.
	\item[(v)] there exists a function $r_B\colon \real_+\times \real_+\to \real$ such that
	\begin{equation*}
	\langle B(u,w),w\rangle_E-H^0(\delta_2 u,\gamma_2w;-\gamma_2w)\ge r_B(\|w\|_E,\|u\|_V)\|w\|_E\mbox{ for all $u\in V$ and $w\in E$},
	\end{equation*}
	and
	\begin{itemize}
		\item[$\bullet$] for every nonempty and bounded set $O\subset \real_+$, we have $r_B(t,s)\to +\infty$ as $t\to +\infty$ for all $s\in O$,
		\item[$\bullet$] for any constants $c_1$, $c_2\ge 0$, it holds $r_B(t,c_1t+c_2)\to +\infty$ as $t\to +\infty$,
		\item[$\bullet$] for sequences $\{s_n\}\subset \real_+$ and $\{t_n\}\subset$  such that
		\begin{equation*}
		s_n\to +\infty,\mbox{ $t_n\to+\infty$ and $\frac{t_n}{s_n}\to 0$ as $n\to\infty$},
		\end{equation*}
		we have
		\begin{equation*}
		r_B(s_n,t_n)\to +\infty\mbox{ as $n\to\infty$}.
		\end{equation*}
	\end{itemize}
\end{itemize}

\begin{theorem}\label{theorems1}
	Under the assumptions $H(A)$, $H(B)$, $H(0)$, $H(1)$,  $H(2)$, $H(J)$, $H(H)$, $H(\psi)$ and $H(\theta)$, the set of solutions to Problem~$\ref{problems1}$,
	denoted by ${\mathbb S}(h,l)$,
	is nonempty and weakly compact in $V\times E$.
\end{theorem}

\begin{proof}
	The proof of this theorem is divided into five steps.
	
	\noindent
	{\bf Step 1.} {\it If the set ${\mathbb S}(h,l)$ of solutions to Problem~$\ref{problems1}$
		is nonempty, then ${\mathbb S}(h,l)$ is bounded.}
	
	Assume that ${\mathbb S}(h,l)$  is nonempty. Let $(u,w)\in {\mathbb S}(h,l)$, $u_0\in \mbox{dom}\psi\cap C$ and $w_0\in \mbox{dom}\theta\cap D$ be arbitrary fixed.  Then, we have
	\begin{equation*}
	\langle A(w,u),u_0-u\rangle_V+J^0(\delta_1 w,\gamma_1 u;\gamma_1 (u_0-u))+\psi(u_0)-\psi(u)\ge \langle h,u_0-u\rangle_V,
	\end{equation*}
	and
	\begin{equation*}
	\langle B(u,w),w_0-w\rangle_E+H^0(\delta_2 u,\gamma_2 w;\gamma_2 (w_0-w))+\theta(w_0)-\theta(w)\ge \langle l,w_0-w\rangle_E.
	\end{equation*}
	From the subadditivity of $x\mapsto J^0(w,u;x)$ and
	$x\mapsto H^0(u,w;x)$, we have
	\begin{align*}
	&\langle A(w,u),u\rangle_V-J^0(\delta_1 w,\gamma_1 u;-\gamma_1 u)\\
	\le & \langle A(w,u),u_0\rangle_V+J^0(\delta_1 w,\gamma_1 u;\gamma_1 u_0)+\psi(u_0)-\psi(u)- \langle h,u_0-u\rangle_V,
	\end{align*}
	and
	\begin{align*}
	&\langle B(u,w),w\rangle_E-H^0(\delta_2 u,\gamma_2 w;-\gamma_2 w)\\
	\le & \langle B(u,w),w_0\rangle_E+H^0(\delta_2 u,\gamma_2 w;\gamma_2 w_0)+\theta(w_0)-\theta(w)- \langle l,w_0-w\rangle_E.
	\end{align*}
	Let $\xi\in X^*$ and $\eta\in Y^*$ be such that
	\begin{align*}
	&\xi\in\partial J(\delta_1w,\gamma_1 u)
	\ \mbox{  and  }\
	\langle\xi,\gamma_1 u_0\rangle_{X}=J^0(\delta_1 w,\gamma_1 u;\gamma_1 u_0),\\
	&\eta\in\partial H(\delta_2u,\gamma_2 w)
	\ \mbox{  and  } \ \langle\eta,\gamma_2 w_0\rangle_{Y}=H^0(\delta_2 u,\gamma_2 w;\gamma_1 w_0).
	\end{align*}
	Recall that $\psi$ and $\theta$ are convex and l.s.c., so, we invoke~\cite[Proposition~5.2.25]{DMP1}
	to find constants $\alpha_\psi$, $\alpha_\theta$, $\beta_\psi$, $\beta_\theta\ge 0$ such that
	\begin{equation*}
	\psi(u)\ge -\alpha_\psi\|u\|_V-\beta_\psi
	\ \mbox{ and }\
	\theta(w)\ge -\alpha_\theta\|w\|_E-\beta_\theta
	\end{equation*}
	for all $(u,w)\in V\times E$.
	We apply the inequalities above and hypotheses $H(A)$(iv)--(v) and  $H(B)$(iv)--(v) to infer that
	\begin{align*}
	&r_A(\|u\|_V,\|w\|_E)\|u\|_V\nonumber\\
	\le &\langle A(w,u),u\rangle_V-J^0(\delta_1 w,\gamma_1 u;-\gamma_1 u)\nonumber\\
	\le &\langle A(w,u),u_0\rangle_V+\langle\xi,\gamma_1 u_0\rangle_{X}+\psi(u_0)-\psi(u)- \langle h,u_0-u\rangle_V\nonumber\\
	\le &b_A(1+\|u\|_V+\|w\|_E)\|u_0\|_V+ c_J(1+\|\gamma_1u\|_X+\|\delta_1w\|_{Z_1})\|\gamma _1u_0\|_X\nonumber\\
	&+\psi(u_0)-\psi(u)+\|h\|_{V^*}\left(\|u_0\|_V+\|u\|_V\right)\nonumber\\
	\le &b_A(1+\|u\|_V+\|w\|_E)\|u_0\|_V+ c_J(1+\|\gamma_1\|\|u\|_V+\|\delta_1\|\|w\|_E)\|\gamma_1\|\|u_0\|_V\nonumber\\
	&+\psi(u_0)+\alpha_\psi\|u\|_V+\beta_\psi+\|h\|_{V^*}\left(\|u_0\|_V+\|u\|_V\right),
	\end{align*}
	and
	\begin{align*}
	&r_B(\|w\|_E,\|u\|_V)\|w\|_E\nonumber\\
	\le &\langle B(u,w),w\rangle_E-H^0(\delta_2 u,\gamma_2 w;-\gamma_2 w)\nonumber\\
	\le &\langle B(u,w),w_0\rangle_E+H^0(\delta_2 u,\gamma_2 w;\gamma_w w_0)+\theta(w_0)-\theta(w)- \langle l,w_0-w\rangle_E\nonumber\\
	\le &b_B(1+\|u\|_V+\|w\|_E)\|w_0\|_E+ c_H(1+\|\delta_2\|\|u\|_V+\|\gamma_2\|\|w\|_E)\|\gamma_2\|\|w_0\|_E\nonumber\\
	&+\theta(w_0)+\alpha_\theta\|w\|_E+\beta_\theta+\|l\|_{E^*}\left(\|w_0\|_E+\|w\|_E\right).
	\end{align*}
	Hence, we have
	\begin{align}\label{eqn3}
	&r_A(\|u\|_V,\|w\|_E) \nonumber\\
	\le &\frac{b_A(1+\|u\|_V+\|w\|_E)\|u_0\|_V+ c_J(1+\|\gamma_1\|\|u\|_V+\|\delta_1\|\|w\|_E)\|\gamma_1\|\|u_0\|_V}{\|u\|_V}\nonumber\\
	&+\frac{\psi(u_0)+\alpha_\psi\|u\|_V+\beta_\psi+\|h\|_{V^*}\left(\|u_0\|_V+\|u\|_V\right)}{\|u\|_V},
	\end{align}
	and
	\begin{align}\label{eqn4}
	&r_B(\|w\|_E,\|u\|_V)\nonumber\\
	\le &\frac{b_B(1+\|u\|_V+\|w\|_E)\|w_0\|_E+ c_H(1+\|\delta_2\|\|u\|_V+\|\gamma_2\|\|w\|_E)\|\gamma_2\|\|w_0\|_E}{\|w\|_E}\nonumber\\
	&+\frac{\theta(w_0)+\alpha_\theta\|w\|_E+\beta_\theta+\|l\|_{E^*}\left(\|w_0\|_E+\|w\|_E\right)}{\|w\|_E}.
	\end{align}
	
	Assume that ${\mathbb S}(h,l)$ is unbounded. Without any loss of generality, we may suppose that there exists a sequence
	$\{(u_n,w_n)\}\subset {\mathbb S}(h,l)$ satisfying $\|u_n\|_V+\|w_n\|_E\to \infty$ as $n\to\infty$, namely, one of the following conditions holds:
	\begin{equation}\label{eqn5}
	\|u_n\|_V\uparrow+\infty\mbox{ as $n\to\infty$ and $\{w_n\}$ is bounded in $E$},
	\end{equation}
	or
	\begin{equation}\label{eqn6}
	\|w_n\|_E\uparrow+\infty\mbox{ as $n\to\infty$ and $\{u_n\}$ is bounded in $V$},
	\end{equation}
	or
	\begin{equation}\label{eqn7}
	\|w_n\|_E\uparrow+\infty\mbox{ and $\|u_n\|_V\uparrow+\infty$ as $n\to\infty$}.
	\end{equation}
	If (\ref{eqn5}) is true, then from (\ref{eqn3}) we obtain
	\begin{align*}
	&r_A(\|u_n\|_V,\|w_n\|_E) \nonumber\\
	\le &\frac{b_A(1+\|u_n\|_V+\|w_n\|_E)\|u_0\|_V+ c_J(1+\|\gamma_1\|\|u_n\|_V+\|\delta_1\|\|w_n\|_E)\|\gamma_1\|\|u_0\|_V}{\|u_n\|_V}\nonumber\\
	&+\frac{\psi(u_0)+\alpha_\psi\|u_n\|_V+\beta_\psi+\|h\|_{V^*}\left(\|u_0\|_V+\|u_n\|_V\right)}{\|u_n\|_V}.
	\end{align*}
	Taking the limit as $n\to\infty$ in the inequality above and using assumption $H(A)$(v),
	it yields
	\begin{align*}
	&+\infty=\lim_{n\to\infty}r_A(\|u_n\|_V,\|w_n\|_E) \nonumber\\
	\le &\lim_{n\to\infty}\bigg[\frac{b_A(1+\|u_n\|_V+\|w_n\|_E)\|u_0\|_V+ c_J(1+\|\gamma_1\|\|u_n\|_V+\|\delta_1\|\|w_n\|_E)\|\gamma_1\|\|u_0\|_V}{\|u_n\|_V}\nonumber\\
	&+\frac{\psi(u_0)+\alpha_\psi\|u_n\|_V+\beta_\psi+\|h\|_{V^*}\left(\|u_0\|_V+\|u_n\|_V\right)}{\|u_n\|_V}\bigg]\\
	=&\, b_A\|u_0\|_V+c_J\|\gamma_1\|^2\|u_0\|_V+\alpha_\psi+\|h\|_{V^*}.
	\end{align*}
	This leads to a contradiction. So, we conclude that ${\mathbb S}(h,l)$ is bounded. Likewise, we can employ the same argument to obtain a contradiction when (\ref{eqn6}) occurs.
	If (\ref{eqn7}) holds, we distinguish further the following cases:
	\begin{itemize}
		\item[i)] $\frac{\|u_n\|_V}{\|w_n\|_E}\to +\infty$ or $\frac{\|w_n\|_E}{\|u_n\|_E}\to +\infty$ as $n\to\infty$.
		\item[ii)] there exists $n_0\in\mathbb N$ such that $0<c_0\le \frac{\|u_n\|_V}{\|w_n\|_E} \le c_1$ for all $n\ge n_0$ for some $c_0,c_1>0$.
	\end{itemize}
	Concerning the case i), we only examine the situation  if
	$\frac{\|u_n\|_V}{\|w_n\|_E}\to +\infty$, because the same conclusion can be obtained by using a similar proof when $\frac{\|w_n\|_E}{\|u_n\|_E}\to +\infty$ as $n\to\infty$.
	Keeping in mind (\ref{eqn3}), one has
	\begin{align*}
	&r_A(\|u_n\|_V,\|w_n\|_E) \nonumber\\
	\le &\frac{b_A(1+\|u_n\|_V+\|w_n\|_E)\|u_0\|_V+ c_J(1+\|\gamma_1\|\|u_n\|_V+\|\delta_1\|\|w_n\|_E)\|\gamma_1\|\|u_0\|_V}{\|u_n\|_V}\nonumber\\
	&+\frac{\psi(u_0)+\alpha_\psi\|u_n\|_V+\beta_\psi+\|h\|_{V^*}\left(\|u_0\|_V+\|u_n\|_V\right)}{\|u_n\|_V}.
	\end{align*}
	By virtue of (\ref{eqn7}) and the condition i),
	we infer
	\begin{align*}
	&+\infty=\lim_{n\to\infty}r_A(\|u_n\|_V,\|w_n\|_E) \nonumber\\
	\le &\lim_{n\to\infty}\bigg[\frac{b_A(1+\|u_n\|_V+\|w_n\|_E)\|u_0\|_V+ c_J(1+\|\gamma_1\|\|u_n\|_V+\|\delta_1\|\|w_n\|_E)\|\gamma_1\|\|u_0\|_V}{\|u_n\|_V}\nonumber\\
	&+\frac{\psi(u_0)+\alpha_\psi\|u_n\|_V+\beta_\psi+\|h\|_{V^*}\left(\|u_0\|_V+\|u_n\|_V\right)}{\|u_n\|_V}\bigg]\\
	=&\, b_A\|u_0\|_V+c_J\|\gamma_1\|^2\|u_0\|_V+\alpha_\psi+\|h\|_{V^*}.
	\end{align*}
	This leads to a contradiction too.
	This implies that ${\mathbb S}(h,l)$ is bounded. Moreover, if the condition ii) holds,
	then we can use hypothesis $H(B)$(v) to get
	\begin{align*}
	&+\infty=\lim_{n\to\infty}r_B(\|w_n\|_E,\|u_n\|_V)\nonumber\\
	\le &\frac{b_B(1+\|u_n\|_V+\|w_n\|_E)\|w_0\|_E+ c_H(1+\|\delta_2\|\|u_n\|_V+\|\gamma_2\|\|w_n\|_E)\|\gamma_2\|\|w_0\|_E}{\|w_n\|_E}\nonumber\\
	&+\frac{\theta(w_0)+\alpha_\theta\|w_n\|_E+\beta_\theta+\|l\|_{E^*}\left(\|w_0\|_E+\|w_n\|_E\right)}{\|w_n\|_E}\\
	\le &\, b_B(c_1+1)\|w_0\|_E+ c_H(\|\delta_2\|c_1+\|\gamma_2\|)\|\gamma_2\|\|w_0\|_E+ \theta(w_0)+\alpha_\theta  +\|l\|_{E^*}.
	\end{align*}
	This implies that the set ${\mathbb S}(h,l)$ is bounded.

	\noindent
	{\bf Step 2.} {\it For each $w\in E$ (resp. $u\in V$) fixed, the solution set of inequality problem $(\ref{eqn1})$ (resp. $(\ref{eqn2})$) is nonempty, bounded and closed.}
	
	Let $w\in E$ be arbitrary fixed. By the definition of generalized subgradient in the sense of Clarke, we have
	\begin{align*}
	&\langle A(w,u)+\gamma^*_1\xi,u\rangle_V=\langle A(w,u),u\rangle_V-\langle \xi,-\gamma_1 u\rangle_X\\
	\ge& \langle A(w,u),u\rangle_V-J^0(\delta_1 y,\gamma_1 u;-\gamma_1u)
	\end{align*}
	for all $\xi\in \partial J(\delta_1w,\gamma_1 u)$. Taking into account hypothesis $H(A)$(v) and the inequality above, it gives
	\begin{align*}
	&\frac{\inf_{\xi\in\partial J(\delta_1 w,\gamma_1u)}\langle A(w,u)+\gamma^*\xi,u\rangle_V }{\|u\|_V}\ge \frac{\langle A(w,u),u\rangle_V-J^0(\delta_1 y,\gamma_1 u;-\gamma_1u)}{\|u\|_V}\\
	\ge&\, r_A(\|u\|_V,\|w\|_E)
	\end{align*}
	for all $u\in V$.
	This means that multivalued operator
	$$
	V\ni u\mapsto A(w,u)+
	\gamma_1^*\partial J(\delta_1w, \gamma_1u)\subset V^*
	$$
	is coercive. In an analogous way, we can verify that for every $u\in V$ fixed,  multivalued operator
	$$
	E\ni w\mapsto B(u,w)+\gamma_2^*\partial H(\delta_2u, \gamma_2w)\subset E^*
	$$
	is coercive as well. Therefore, we can invoke the same arguments as in the proof
	of~\cite[Theorem~3]{Liu-Liu-Wen-Yao-Zeng-2020-AMO} to conclude that the solution set of inequality problem (\ref{eqn1}) (resp. (\ref{eqn2})) is nonempty, bounded and closed.
	
	Next, we introduce the multivalued map
	$\Gamma\colon C\times D\to 2^{C\times D}$
	defined by
	\begin{equation}\label{eqn8}
	\Gamma(u,w):=(\mathcal P(w),\mathcal Q(u))
	\ \mbox{ for all } \ (u,w)\in C\times D,
	\end{equation}
	where $\mathcal P\colon E\to 2^C$ and $\mathcal Q\colon V\to 2^D$ stand for the solution mappings  problems (\ref{eqn1}) and (\ref{eqn2}), respectively, namely, $\mathcal P(w)$ and $\mathcal Q(u)$ are the solution sets of problems (\ref{eqn1}) and (\ref{eqn2}) corresponding to $w\in E$ and $u\in V$, respectively.
	From the definition of $\Gamma$, it is not difficult to show that $(u,w)\in C\times D$ is a fixed point of $\Gamma$ if and only if it is a solution of Problem~\ref{problems1}. Based on this property, we are going to verify that $\Gamma$ has at least one fixed point in $C\times D$.

	\noindent
	{\bf Step 3.} {\it There exists a bounded, closed and convex subset $\mathcal X$ of $C\times D$ such that $\Gamma$ maps $\mathcal X$ into itself.}
	
	Indeed, it is sufficient to show that there exists a constant $m_0>0$ such that
	\begin{equation}\label{eqn9}
	\Gamma(\mathcal O(m_0))\subset \mathcal O(m_0),
	\end{equation}
	where $\mathcal O(m_0)$ is defined by
	\begin{equation*}
	\mathcal O(m_0):=\{(u,w)\in C\times D\,\mid\,\|u\|_V\le m_0\mbox{ and }\|w\|_E\le m_0\}.
	\end{equation*}
	Arguing by contradiction, there is no $m_0>0$ such that (\ref{eqn9}) holds. So, for each $n\in \mathbb N$, there exist sequences $(u_n,w_n)$, $(v_n,z_n)\in C\times D$ satisfying
	\begin{equation*}
	(u_n,w_n)\in \mathcal O(n),\mbox{  $(v_n,z_n)\in \Gamma(u_n,w_n)$ and $\|v_n\|_V>n$ or $\|z_n\|_E>n$}.
	\end{equation*}
	Now, we suppose that $\|v_n\|_V>n$.
	It follows from (\ref{eqn3}) that
	\begin{align*}
	&r_A(\|v_n\|_V,\|w_n\|_E) \nonumber\\
	\le &\frac{b_A(1+\|v_n\|_V+\|w_n\|_E)\|u_0\|_V+ c_J(1+\|\gamma_1\|\|v_n\|_V+\|\delta_1\|\|w_n\|_E)\|\gamma_1\|\|u_0\|_V}{\|v_n\|_V}\nonumber\\
	&+\frac{\psi(u_0)+\alpha_\psi\|v_n\|_V+\beta_\psi+\|h\|_{V^*}\left(\|u_0\|_V+\|v_n\|_V\right)}{\|v_n\|_V}.
	\end{align*}
	Recalling that $\|w_n\|_E\le n<\|v_n\|_V$,
	we apply hypothesis $H(A)$(v) to find that
	\begin{align*}
	&+\infty=\lim_{n\to\infty}r_A(\|v_n\|_V,\|w_n\|_E) \nonumber\\
	\le &\lim_{n\to\infty}\bigg[\frac{b_A(1+\|v_n\|_V+\|w_n\|_E)\|u_0\|_V+ c_J(1+\|\gamma_1\|\|v_n\|_V+\|\delta_1\|\|w_n\|_E)\|\gamma_1\|\|u_0\|_V}{\|v_n\|_V}\nonumber\\
	&+\frac{\psi(u_0)+\alpha_\psi\|v_n\|_V+\beta_\psi+\|h\|_{V^*}\left(\|u_0\|_V+\|v_n\|_V\right)}{\|v_n\|_V}\bigg]\\
	\le &\, 2b_A\|u_0\|_V+c_J(\|\gamma_1\|+\|\delta_1\|)\|\gamma_1\|\|u_0\|_V+\alpha_\psi+\|h\|_{V^*}.
	\end{align*}
	Hence we get a contradiction.
	As we did before, it can also lead to a contraction when the case $\|z_n\|_E>n$ occurs. This means that there exists a constant $m_0>0$ such that (\ref{eqn9}) is valid. Therefore, we conclude that there exists a bounded, closed and convex subset $\mathcal X$ of $C\times D$ such that $\Gamma$ maps $\mathcal X$ into itself.

	\noindent
	{\bf Step 4.} {\it $\Gamma$ is weakly-weakly upper semicontinuous.}
	
	Let $\mathcal M\subset C\times D$ be an arbitrary weakly closed set such that
	$\Gamma^-(\mathcal M)\neq\emptyset$.
	From Proposition~\ref{uscproposition}, it is sufficient to verify that $\Gamma^-(\mathcal M)$ is weakly closed in $V\times E$. Let $\{(v_n,z_n)\}\subset \Gamma^-(\mathcal M)$ be such that
	\begin{equation}\label{eqn10}
	(u_n,w_n)\wto (u,w)\mbox{ in $V\times E$ as $n\to\infty$}
	\end{equation}
	for some $(u,w)\in V\times E$. Hence, for every $n\in\mathbb N$, we are able to find $(u_n,w_n)\in V\times E$ satisfying
	\begin{equation*}
	(v_n,z_n)\in \Gamma(u_n,w_n)\cap\mathcal M.
	\end{equation*}
	From (\ref{eqn3}) and (\ref{eqn4}), we have
	\begin{align*}
	&r_A(\|v_n\|_V,\|w_n\|_E) \nonumber\\
	\le &\frac{b_A(1+\|v_n\|_V+\|w_n\|_E)\|u_0\|_V+ c_J(1+\|\gamma_1\|\|v_n\|_V+\|\delta_1\|\|w_n\|_E)\|\gamma_1\|\|u_0\|_V}{\|v_n\|_V}\nonumber\\
	&+\frac{\psi(u_0)+\alpha_\psi\|v_n\|_V+\beta_\psi+\|h\|_{V^*}\left(\|u_0\|_V+\|v_n\|_V\right)}{\|v_n\|_V},
	\end{align*}
	and
	\begin{align*}
	&r_B(\|z_n\|_E,\|u_n\|_V)\nonumber\\
	\le &\frac{b_B(1+\|u_n\|_V+\|z_n\|_E)\|w_0\|_E+ c_H(1+\|\delta_2\|\|u_n\|_V+\|\gamma_2\|\|z_n\|_E)\|\gamma_2\|\|w_0\|_E}{\|z_n\|_E}\nonumber\\
	&+\frac{\theta(w_0)+\alpha_\theta\|z_n\|_E+\beta_\theta+\|l\|_{E^*}\left(\|w_0\|_E+\|z_n\|_E\right)}{\|z_n\|_E}.
	\end{align*}
	Combining the latter with $H(A)$(v), $H(B)$(v), we infer that sequence $\{(v_n,$ $z_n)\}$ is bounded in $V\times E$. Passing to a relabeled subsequence if necessary, we may suppose that
	\begin{equation}\label{eqn11}
	(v_n,z_n)\wto (v,z)\mbox{ in $V\times E$ as $n\to\infty$}
	\end{equation}
	for some $(v,z)\in V\times E$.
	
	Next, for each $n\in\mathbb N$, let $\xi_n\in X^*$ and $\eta_n\in Y^*$ be such that
	\begin{align*}
	&\langle\xi_n,\gamma_1 (x-v_n)\rangle_X=J^0(\delta_1 w_n,\gamma_1 v_n;\gamma_1(x-v_n)),\\
	&\langle\eta_n,\gamma_2 (y-z_n)\rangle_Y=H^0(\delta_2 u_n,\gamma_2 z_n;\gamma_2(y-z_n)).
	\end{align*}
	Hence, we have
	\begin{equation*}
	\langle A(w_n,v_n)+\gamma _1^*\xi_n,x-v_n\rangle_V+\psi(x)-\psi(v_n)\ge \langle h,x-v_n\rangle_V
	\end{equation*}
	for all $x\in C$, and
	\begin{equation*}
	\langle B(u_n,z_n)+\gamma_2^*\eta_n,y-z_n\rangle_E+\theta(y)-\theta(z_n)\ge \langle l,y-z_n\rangle_E
	\end{equation*}
	for all $y\in D$.
	Then, we use the hypotheses $H(A)$(ii) and $H(B)$(ii) to find that
	\begin{align*}
	&\langle h,x-v_n\rangle_V\\
	\le &
	\langle A(w_n,x)+\gamma _1^*\alpha_n,x-v_n\rangle_V+\psi(x)-\psi(v_n)\\
	\le& \langle A(w_n,x),x-v_n\rangle_V+J^0(\delta_1 w_n,\gamma_1x;\gamma_1(x-v_n))+\psi(x)-\psi(v_n)
	\end{align*}
	for all $\alpha_n\in\partial J(\delta_1 w_n,\gamma_1 x)$ and $x\in C$, and
	\begin{align*}
	& \langle l,y-z\rangle_E\\
	\le &
	\langle B(u_n,y)+\gamma_2^*\beta_n,y-z_n\rangle_E+\theta(y)-\theta(z_n)\\
	\le &\langle B(u_n,y),y-z_n\rangle_E+H^0(\delta_2 u_n,\gamma_2 y;\gamma_2(y-z_n))+\theta(y)-\theta(z_n)
	\end{align*}
	for all $\beta_n\in\partial H(\delta_2u_n,\gamma_2 y)$ and $y\in D$. Passing to the upper limit as $n\to\infty$ in the inequalities above and using hypotheses $H(J)$(iii), $H(H)$(iii), $H(A)$(iii) and $H(B)$(iii), we obtain
	\begin{align*}
	&\langle h,x-v\rangle_V\\
	=&\limsup_{n\to\infty}\langle h,x-v_n\rangle_V\\
	\le& \limsup_{n\to\infty}\left[\langle A(w_n,x),x-v_n\rangle_V+J^0(\delta_1 w_n,\gamma_1x;\gamma_1(x-v_n))+\psi(x)-\psi(v_n)\right]\\
	\le& \limsup_{n\to\infty}\langle A(w_n,x),x-v_n\rangle_V+ \limsup_{n\to\infty}J^0(\delta_1 w_n,\gamma_1x;\gamma_1(x-v_n))\\
	&+\psi(x)- \liminf_{n\to\infty}\psi(v_n) \\
	\le &\langle A(w,x),x-v\rangle_V+J^0(\delta_1 w,\gamma_1x;\gamma_1(x-v))+\psi(x)-\psi(v),
	\end{align*}
	and
	\begin{align*}
	&\langle l,y-z_n\rangle_E\\
	=&\limsup_{n\to\infty} \langle l,y-z_n\rangle_E\\
	\le &\limsup_{n\to\infty}\left[\langle B(u_n,y),y-z_n\rangle_E+H^0(\delta_2 u_n,\gamma_2 y;\gamma_2(y-z_n))+\theta(y)-\theta(z_n)\right]\\
	\le& \limsup_{n\to\infty}\langle B(u_n,y),y-z_n\rangle_E+\limsup_{n\to\infty} H^0(\delta_2 u_n,\gamma_2 y;\gamma_2(y-z_n))\\
	&+\theta(y)-\liminf_{n\to\infty}\theta(z_n)\\
	\le &\langle B(u,y),y-z\rangle_E+H^0(\delta_2 u,\gamma_2 y;\gamma_2(y-z))+\theta(y)-\theta(z).
	\end{align*}
	Hence
	\begin{equation}\label{eqn001}
	\langle A(w,x),x-v\rangle_V+J^0(\delta_1 w,\gamma_1x;\gamma_1(x-v))+\psi(x)-\psi(v)\ge \langle h,x-v\rangle_V
	\end{equation}
	for all $x\in C$,
	and
	\begin{equation}\label{eqn002}
	\langle B(u,y),y-z\rangle_E+H^0(\delta_2 u,\gamma_2 y;\gamma_2(y-z))+\theta(y)-\theta(z)\ge \langle l,y-z_n\rangle_E
	\end{equation}
	for all $y\in D$.
	Let $t\in(0,1)$ and $r\in C$ be arbitrary. We insert $x=x_t:=tr+(1-t)v$ into (\ref{eqn001}) and apply the positive homogeneity of $v\mapsto J^0(w,u;v)$ and convexity of $\psi$ to get
	\begin{align*}
	&t\left[\langle A(w,x_t),r-v\rangle_V+J^0(\delta_1 w,\gamma_1x_t;\gamma_1(r-v))+\psi(r)-\psi(v)\right]\\
	\ge &t\langle A(w,x_t),r-v\rangle_V+J^0(\delta_1 w,\gamma_1x_t;t\gamma_1(r-v))+\psi(x_t)-\psi(v)\\
	\ge& t\langle h,r-v\rangle_V.
	\end{align*}
	So, we obtain
	\begin{equation*}
	\langle A(w,x_t),r-v\rangle_V+J^0(\delta_1 w,\gamma_1x_t;\gamma_1(r-v))+\psi(r)-\psi(v)\ge\langle h,r-v\rangle_V.
	\end{equation*}
	Passing to the upper limit as $t\downarrow0$ in the inequality above and using hypothesis $H(A)$(i) and upper semicontinuity of $(u,v)\mapsto J^0(w,u;v)$, we have
	\begin{align*}
	&\langle A(w,v),r-v\rangle_V+J^0(\delta_1 w,\gamma_1 v;\gamma_1(r-v))+\psi(r)-\psi(v)\\
	\ge &
	\limsup_{t\downarrow0}\left[ \langle A(w,x_t),r-v\rangle_V+J^0(\delta_1 w,\gamma_1x_t;\gamma_1(r-v))+\psi(r)-\psi(v)\right]\\
	\ge&\langle h,r-v\rangle_V.
	\end{align*}
	Since $r\in C$ is arbitrary, so, we have that $v\in C$ is a solution of problem (\ref{eqn1}) corresponding to $w\in D$. Similarly, we also can obtain that $z\in D$ is a solution of problem (\ref{eqn2}) corresponding to $u\in C$.  This implies that $(v,z)\in \Gamma(u,w)$. Due to the weak closedness of $\mathcal M$, we obtain that $(v,z)\in \mathcal M$, that is, $(u,w)\in \Gamma^-(\mathcal M)$. Therefore, we use  Proposition~\ref{uscproposition} to conclude that  $\Gamma$ is weakly-weakly u.s.c..
	
	We conclude that all conditions of the Tychonoff theorem, Theorem~\ref{TychonoffFPT}, have been verified. Using this theorem, we deduce that
	$\Gamma$ has at least one fixed point
	$(u^*,w^*)\in C\times D$ in $\mathcal X$.
	This implies that $(u^*,w^*)\in C\times D$ is also a solution of Problem~\ref{problems1}.

	\noindent
	{\bf Step 5.} {\it The set ${\mathbb S}(h,l)$ is weakly compact.}
	
	From Step 1, we can see that the set
	${\mathbb S}(h,l)$ is bounded. Because of the reflexivity of $V\times E$, it is sufficient to show that ${\mathbb S}(h,l)$  is weakly closed.
	Let $\{(u_n,w_n)\}\subset {\mathbb S}(h,l)$
	be a sequence such that
	\begin{equation}\label{eqn12}
	(u_n,w_n)\wto (u,w)\mbox{ in $V\times E$ as $n\to\infty$}
	\end{equation}
	for some $(u,w)\in V\times E$. In virtue of the definition of $\Gamma$, it yields $(u_n,w_n)\in \Gamma(u_n,w_n)$. Keeping in mind that $\Gamma$ is weakly-weakly u.s.c. and has nonempty, bounded, closed and convex values, it follows
	from~\cite[Theorem~1.1.4]{Kamemsloo-Obukhovskii-Zecca-2001} that $\Gamma$ is weakly-weakly closed. The latter together with the convergence (\ref{eqn12}) implies that $(u,v)\in \Gamma(u,v)$. From the definition of $\Gamma$, we infer that $(u,v)\in {\mathbb S}(h,l)$, i.e., ${\mathbb S}(h,l)$ is weakly closed. Consequently, we conclude that ${\mathbb S}(h,l)$ is weakly compact in $V\times E$.
	This completes the proof.
\end{proof}

We formulate several corollaries of Theorem~\ref{theorems1}.
To this end, we need the following hypotheses.

\medskip

\noindent ${\underline{H(J')}}$:
$J\colon X\to\real$ is locally Lipschitz continuous such that
there exists a constant $c_J\ge 0$ such that
\begin{equation*}
\|\xi\|_{X^*}\le c_J\left(1+\|u\|_X \right)
\end{equation*}
for all $\xi\in \partial J(u)$ and $u\in X$.

\medskip

\noindent ${\underline{H(H')}}$:
$H\colon Y\to\real$ is locally Lipschitz continuous such that
there exists a constant $c_H\ge 0$ such that
\begin{equation*}
\|\eta\|_{Y^*}\le c_H\left(1 +\|w\|_Y\right)
\end{equation*}
for all $\eta\in \partial H(w)$ and $w\in Y$.

\medskip

\noindent  ${\underline{H(2')}}$: $\gamma_1\colon V\to X$ and $\gamma_2\colon E\to Y$ are bounded, linear and compact.

\medskip

\noindent  ${\underline{H(A)}}$(ii)':  for each $w\in E$ the multivalued mapping $V\ni u\mapsto A(w,u)+\gamma_1^*\partial J(\gamma_1u)\subset V^*$ is stable $\psi$-pseudomonotone with respect to $\{h\}$.

\medskip

\noindent  ${\underline{H(A)}}$(v)': there exists a function $r_A\colon \real_+\times \real_+\to \real$ such that
\begin{equation*}
\langle A(w,u),u\rangle_V-J^0(\gamma_1u;-\gamma_1u)\ge r_A(\|u\|_V,\|w\|_E)\|u\|_V\mbox{ for all $u\in V$ and $w\in E$},
\end{equation*}
and
\begin{itemize}
	\item[$\bullet$] for every nonempty and bounded set $O\subset \real_+$, we have $r_A(t,s)\to +\infty$ as $t\to +\infty$ for all $s\in O$,
	\item[$\bullet$] for any constants $c_1,c_2\ge 0$, it holds $r_A(t,c_1t+c_2)\to +\infty$ as $t\to +\infty$,
	\item[$\bullet$] for sequences $\{s_n\}\subset \real_+$ and $\{t_n\}\subset$  such that
	\begin{equation*}
	s_n\to +\infty,\mbox{ $t_n\to+\infty$ and $\frac{t_n}{s_n}\to 0$ as $n\to\infty$},
	\end{equation*}
	we have
	\begin{equation*}
	r_A(s_n,t_n)\to +\infty\mbox{ as $n\to\infty$}.
	\end{equation*}
\end{itemize}

\medskip

\noindent  ${\underline{H(B)}}$(ii)': for each $u\in V$ the multivalued mapping $E\ni w\mapsto B(u,w)+\gamma_2^*\partial H(\gamma_2w)\subset E^*$ is stable $\theta$-pseudomonotone with respect to $\{l\}$.

\medskip

\noindent  ${\underline{H(B)}}$(v)': there exists a function $r_B\colon \real_+\times \real_+\to \real$ such that
\begin{equation*}
\langle B(u,w),w\rangle_E-H^0(\gamma_2w;-\gamma_2w)\ge r_B(\|w\|_E,\|u\|_V)\|w\|_E\mbox{ for all $u\in V$ and $w\in E$},
\end{equation*}
and
\begin{itemize}
	\item[$\bullet$] for every nonempty and bounded set $O\subset \real_+$, we have $r_B(t,s)\to +\infty$ as $t\to +\infty$ for all $s\in O$,
	\item[$\bullet$] for any constants $c_1,c_2\ge 0$, it holds $r_B(t,c_1t+c_2)\to +\infty$ as $t\to +\infty$,
	\item[$\bullet$] for sequences $\{s_n\}\subset \real_+$ and $\{t_n\}\subset$  such that
	\begin{equation*}
	s_n\to +\infty,\mbox{ $t_n\to+\infty$ and $\frac{t_n}{s_n}\to 0$ as $n\to\infty$},
	\end{equation*}
	we have
	\begin{equation*}
	r_B(s_n,t_n)\to +\infty\mbox{ as $n\to\infty$}.
	\end{equation*}
\end{itemize}

\begin{corollary}\label{corollarys1}
	Suppose that $H(A)${\rm(i), (ii)', (iii), (iv)'}, $H(B)${\rm(i), (ii)', (iii), (iv)'}, $H(0)$, $H(1)$,  $H(2')$, $H(J')$, $H(H')$, $H(\psi)$ and $H(\theta)$ hold. Then, the set of solutions to problem $(\ref{prob1a})$--$(\ref{prob1b})$
	is nonempty and weakly compact in $V\times E$.
\end{corollary}

\begin{corollary}\label{corollarys2}
	Suppose that $H(A)$, $H(B)$, $H(0)$, $H(1)$,  $H(2)$, $H(J)$ and $H(H)$ are fulfilled. Then, the set of solutions to problem $(\ref{prob2a})$--$(\ref{prob2b})$
	is nonempty and weakly compact in $V\times E$.
\end{corollary}

\medskip

\noindent  ${\underline{H(A)}}$(ii)'':  for each $w\in E$,  $V\ni u\mapsto A(w,u)\in V^*$ is stable $\psi$-pseudomonotone with respect to $\{h\}$.

\medskip

\noindent  ${\underline{H(A)}}$(v)'': there exists a function $r_A\colon \real_+\times \real_+\to \real$ such that
\begin{equation*}
\langle A(w,u),u\rangle_V \ge r_A(\|u\|_V,\|w\|_E)\|u\|_V\mbox{ for all $u\in V$ and $w\in E$},
\end{equation*}
and
\begin{itemize}
	\item[$\bullet$] for every nonempty and bounded set $O\subset \real_+$, we have $r_A(t,s)\to +\infty$ as $t\to +\infty$ for all $s\in O$,
	\item[$\bullet$] for any constants $c_1,c_2\ge 0$, it holds $r_A(t,c_1t+c_2)\to +\infty$ as $t\to +\infty$,
	\item[$\bullet$] for sequences $\{s_n\}\subset \real_+$ and $\{t_n\}\subset$  such that
	\begin{equation*}
	s_n\to +\infty,\mbox{ $t_n\to+\infty$ and $\frac{t_n}{s_n}\to 0$ as $n\to\infty$},
	\end{equation*}
	we have
	\begin{equation*}
	r_A(s_n,t_n)\to +\infty\mbox{ as $n\to\infty$}.
	\end{equation*}
\end{itemize}

\medskip

\noindent  ${\underline{H(B)}}$(ii)'': for each $u\in V$,  $E\ni w\mapsto B(u,w)\in E^*$ is stable $\theta$-pseudomonotone with respect to $\{l\}$.

\medskip

\noindent  ${\underline{H(B)}}$(v)'': there exists a function $r_B\colon \real_+\times \real_+\to \real$ such that
\begin{equation*}
\langle B(u,w),w\rangle_E \ge r_B(\|w\|_E,\|u\|_V)\|w\|_E\mbox{ for all $u\in V$ and $w\in E$},
\end{equation*}
and
\begin{itemize}
	\item[$\bullet$] for every nonempty and bounded set $O\subset \real_+$, we have $r_B(t,s)\to +\infty$ as $t\to +\infty$ for all $s\in O$,
	\item[$\bullet$] for any constants $c_1,c_2\ge 0$, it holds $r_B(t,c_1t+c_2)\to +\infty$ as $t\to +\infty$,
	\item[$\bullet$] for sequences $\{s_n\}\subset \real_+$ and $\{t_n\}\subset$  such that
	\begin{equation*}
	s_n\to +\infty,\mbox{ $t_n\to+\infty$ and $\frac{t_n}{s_n}\to 0$ as $n\to\infty$},
	\end{equation*}
	we have
	\begin{equation*}
	r_B(s_n,t_n)\to +\infty\mbox{ as $n\to\infty$}.
	\end{equation*}
\end{itemize}

\begin{corollary}\label{corollarys3}
	Suppose that $H(A)${\rm(i), (ii)'', (iii), (iv)''}, $H(B)${\rm(i), (ii)'', (iii), (iv)''}, $H(0)$, $H(1)$, $H(\psi)$ and $H(\theta)$ are satisfied. Then, the set of solutions to problem $(\ref{eqn003})$--$(\ref{eqn004})$
	is nonempty and weakly compact in $V\times E$.
\end{corollary}

\begin{remark}
	Corollary~\ref{corollarys3} coincides with
	a result of Liu-Yang-Zeng-Zhao~\cite[Theorem~7]{Liu-Yang-Zeng-Zhao-2021-JOTA}.
	In comparision with that result,
	in the present paper, we give a new proof which is based on a multivalued version of the Tychonoff fixed point principle in a Banach space along with the theory of nonsmoth analysis, the generalized monotonicity arguments and the Minty approach.
	We conclude that our results are much more general,
	improve the former one in several directions, and
	are proved by using a new approach.
\end{remark}


\section{Conclusions}
In the present paper,
a coupled system which consists of two nonlinear variational-hemi\-va\-ria\-tio\-nal inequalities with constraints in Banach spaces has been investigated.
A general
existence result to the system
was established by using a multivalued version of the Tychonoff fixed point principle in a Banach space together with the theory of nonsmoth analysis, generalized monotonicity arguments and the Minty approach. Our result extends the recent ones
obtained in~\cite[Theorem~7]{Liu-Yang-Zeng-Zhao-2021-JOTA}.

There are plenty of problems arising in engineering applications which can be formulated as a system of coupled variational-hemivariational inequalities.
With this motivation, in the future,
we plan to utilize the theoretical results established in this paper to study various real engineering problems. Also, we will further develop the mathematical theory for systems of the variational-hemivariational inequalities, to cover, for instance, stability analysis, optimal control, sensitivity, and homogenization.

\vskip 6mm
\noindent{\bf Acknowledgments}

\noindent
This project has received funding from the NNSF of China Grant Nos. 12001478 and 12101143, the European Union's Horizon 2020 Research and Innovation Programme under the Marie Sklodowska-Curie grant agreement No. 823731 CONMECH,
and the Startup Project of Doctor Scientific Research of Yulin Normal University No. G2020ZK07.
It is also supported by Natural Science Foundation of Guangxi Grants Nos.
2021GXNSFFA196004,
2020GXNSFBA\-297137,
GKAD21220144,
2018GXNSFAA281353,
the Ministry of Science and Higher Education of Republic of Poland under Grant No. 
440328/PnH2/2019, and the National Science Centre of Poland under Project No.
2021/41/B/ST1/01636.


\end{document}